%% file: main.tex
\documentclass[reqno,12pt]{amsart}

\input{preamble.tex}

\title{On some periodic continued fractions along the $\mathbb{Z}_2$ extension over $\mathbb{Q}$}

\author{Yoshinori Kanamura}
\address{Y. Kanamura NTT DATA Mathematical Systems Inc., Tokyo, Japan}
\email{kana1118yoshi@keio.jp}

\author{Hyuga Yoshizaki}
\address{H. Yoshizaki Department of Mathematics, Faculty of Science and Technology, Tokyo University of Science; 2641 Yamazaki, Noda-shi, Chiba, Japan}
\email{yoshizaki.hyuga@gmail.com}


\date{\today}

\begin{document}
\maketitle

\begin{abstract}
    In 2021, Brock, Elkies, and Jordan generalized the theory of periodic continued fractions (PCFs) over $\mathbb{Z}$ to the ring of integers in a number field.
    In particular, they considered the case where the number field is an intermediate field of the $\mathbb{Z}_2$-extension over $\mathbb{Q}$ and asked whether a $(N, \ell)$-type PCF for $X_n = 2\cos(2\pi/2^{n+2})$ exists.
    In this paper, we construct $(1,2)$ and $(0,3)$-type PCFs for $X_n$ for all $n\geq1$. To the best of our knowledge, this is the first explicit construction of type (0,3) continued fractions for all $n\geq1$.
    To obtain such results, for each type, we construct a bijection between a certain subset of the group of relative units in each layer of the $\mathbb{Z}_2$-extension and the set of PCFs for $X_n$.
    While our result confirms the existence of such PCFs for all $n\geq1$ in types $(1,2)$ and $(0,3)$, determining all PCFs remains an open problem.
    The bijections constructed in our result translate this problem into the study of the subsets of the relative units.
    As a second main result, we give explicit bounds for the logarithms of the relative units corresponding to $(1,2)$ or $(0,3)$-type PCFs for $X_n$.
    These bounds allow us to explain interesting phenomena observed in the distribution of such points.
\end{abstract}

\section{Introduction}\label{introduction}

\input{sec1_intro.tex}

\section{Preliminaries}\label{preliminaries}

\input{sec2_prelim.tex}

\section{Key Theorem}\label{key_theorem}

\input{sec3_key_thm.tex}

\section{Proof of \cref{maintheorem1}}\label{main_thm_1}

\input{sec4_main_thm_1.tex}

\section{Proof of \cref{maintheorem2}}\label{main_thm_2}

\input{sec5_main_thm_2.tex}
\section{Relationship with the generalized Pell equation and \cref{maintheorem1_explicit}}\label{rel_to_pell_eq}

\input{sec6_rel_to_pell_eq.tex}

\section*{Acknowledgements}
The authors grateful to Yoshinosuke Hirakawa for his careful reading and valuable comments on a draft of this paper.
The authors also would like to thank Prof. Tomokazu Kashio and Prof. Shuji Yamamoto for their valuable comments.

\bibliographystyle{alpha}
\bibliography{references}

\end{document}

%% file: preamble.tex
\usepackage[pdftex]{graphicx}
\pdfoutput=1
\usepackage[top=30truemm,bottom=30truemm,left=30truemm,right=30truemm]{geometry} 
\usepackage{amssymb}
\usepackage{amsthm}
\usepackage{amsmath}
\usepackage{color}
\usepackage[driverfallback=dvipdfm]{hyperref}
\usepackage{cleveref}
\usepackage{tikz}
\usepackage{mdframed}
\usepackage{cases}

\newcommand{\ctext}[1]{\raise0.2ex\hbox{\textcircled{\scriptsize{#1}}}}

\newcommand{\Q}{\mathbb{Q}}
\newcommand{\R}{\mathbb{R}}
\newcommand{\B}{\mathbb{B}}
\newcommand{\Z}{\mathbb{Z}}
\newcommand{\C}{\mathbb{C}}
\newcommand{\PP}{\mathbb{P}}
\newcommand{\e}{\varepsilon}

\newcommand{\s}{\sigma}

\newcommand{\G}{\text{Gal}}
\newcommand{\Gal}{\text{Gal}}

\newcommand{\Tr}{\text{Tr}}
\newcommand{\sgn}{\text{sgn}}
\newcommand{\setmid}{\mathrel{}\middle|\mathrel{}}

\theoremstyle{definition}
 \newtheorem{theorem}{Theorem}[section]
 \crefname{theorem}{Theorem}{Theorems}
 \newtheorem{proposition}[theorem]{Proposition}
 \crefname{proposition}{Proposition}{Propositions}
 \newtheorem{lemma}[theorem]{Lemma}
 \crefname{lemma}{Lemma}{Lemmas}
 
 \crefname{corollary}{Corollary}{Corollaries}
 
 \crefname{conjecture}{Conjecture}{Conjectures}
 
 \crefname{question}{Question}{Questions}
 \newtheorem{problem}[theorem]{Problem}
 \crefname{problem}{Problem}{Problems}
 
 \crefname{remark}{Remark}{Remarks}

\theoremstyle{definition} 
 \newtheorem{definition}[theorem]{Definition}
 \crefname{definition}{Definition}{Definitions}
 
 \crefname{example}{Example}{Examples}
 
 \crefname{caution}{Caution}{Cautions}
 
 \crefname{equation}{formula}{formulas}

%% file: sec1_intro.tex
Let $(a_n)_{n\geq 0}$ be a sequence of integers such that $a_i \geq 1$ for $i \geq 1$. 
We write an infinite regular continued fraction as
\begin{align}\label{notationofRCF}
    [a_0, a_1, a_2 \dots] = a_0 + \displaystyle\frac{1}{a_1 + \displaystyle\frac{1}{a_2 + \dots}}.
\end{align}
If there exist $N\in\Z_{\geq 0}$ and $\ell\in\Z_{\geq 1}$ such that $a_{n+\ell}=a_{n}$ for all $n\geq N$, we call (\ref{notationofRCF}) a periodic regular continued fraction and denote it by
\begin{align}\label{periodicnotation}
    P 
    &:= [a_0, \dots ,a_{N-1}, \overline{a_{N}, \dots, a_{N+\ell-1}}]\notag\\
    &= [a_0, \dots, a_{N-1}, a_{N}, \dots ,a_{N+\ell-1}, a_{N}, \dots ,a_{N+\ell-1}, \dots].
\end{align}
In particular,
we say that $P$ has type $(N, \ell)$ and period $\ell$ if $N$ and $\ell$ are both minimal.
It is well-known that every quadratic irrational number has a unique representation as a periodic regular continued fraction.

In this paper, 
we generalize $a_n$ in (\ref{periodicnotation}) from $\Z$ to $\Z[X_{n}]$ where $X_n = 2\cos(2\pi/2^{n+2})$, and call it $\Z[X_{n}]$-PCFs.
We note that $\B_{n} = \Q(X_n)$ is the $n$-th layer of the $\Z_2$-extension over $\Q$ and the ring of integers of $\B_n$ is $\Z[X_n]$.
Recently, Brock--Elkies--Jordan~\cite{BrockElkiesJordan2021} asked the following problem.
\begin{problem}(cf.\ \cite[Problem 1.1]{BrockElkiesJordan2021})\label{BEJproblem}
    For every $n, \ell \in \Z_{\geq 1}$ and $N\in \Z_{\geq 0}$,
    find $\Z[X_{n-1}]$-PCFs of type $(N,\ell)$ for $X_{n}$.
\end{problem}
Brock--Elkies--Jordan solved this problem for some types in $n=1$ and $2$ (cf. \cite[p.381, Table 1.]{BrockElkiesJordan2021}).
Moreover, they showed that there is no $\Z[X_{n-1}]$-PCFs of type $(0,1), (0,2)$, and $(1,1)$ for all $n\in \Z_{\geq 2}$ (\cite[Proposition 5.1, 5.2 (b), 5.3 (b)]{BrockElkiesJordan2021}). 
Although it seems hard to find $\Z[X_{n-1}]$-PCFs for $n\in \Z_{\geq 3}$,
the second author~\cite{Yoshizaki2023} introduced a new continued fraction expansion algorithm and gave a $\Z[X_{n-1}]$-PCF for each $n\in\Z_{\geq 1}$ when $(N,\ell) = (1,2)$.

In this paper, we obtain the following theorem. 
\begin{theorem}\label{maintheorem1} (\cref{maintheorem1_explicit})
    For each $n\in\Z_{\geq 2}$, 
    we obtain $\Z[X_{n-1}]$-PCFs of type $(1,2)$ and $(0,3)$ for $X_{n}$.
\end{theorem}
In other words, 
we answer to \cref{BEJproblem} for all $n\in \Z_{\geq 2}$ when $(N,\ell) = (1,2)$ and $(0,3)$.
Remark that a $\Z[X_{n-1}]$-PCF of type $(N,\ell)$ for $X_{n}$ is not unique.
In fact, our $\Z[X_{n-1}]$-PCFs of type $(1,2)$ are different from ones which the second author~\cite{Yoshizaki2023} obtained.
To prove this theorem,
we give a bijection $P^{1,2}_{n}$ (resp.\ $P^{0,3}_{n}$) between the set $RE_{n}^{+}(1,2)$ (resp. $RE_{n}^{-}(0,3)$) of certain units in $\Z[X_{n}]$ and the set of all $\Z[X_{n-1}]$-PCFs of type $(1,2)$ (resp.\ $(0,3)$) for $X_{n}$ (for more details, see \Cref{key_theorem}).

When $(N, \ell)=(1,2)$ and $(0,3)$,
Brock--Elkies--Jordan did not only solve \cref{BEJproblem} but also determined all $\Z[X_1]$-PCFs for $X_{2}$~\cite[Corollary 6.8 and Corollary 8.19]{BrockElkiesJordan2021}.
They further showed that the number of $\Z[X_{n-1}]$-PCFs for $X_n$ is finite for each $n\geq3$ by using Siegel's theorem (\cite[Theorem 6.3 and Theorem 8.1 (b)]{BrockElkiesJordan2021}).
However, to the best of our knowledge, there may be no results which determine all $\Z[X_{n-1}]$-PCFs for $X_n$ for $n\geq 3$.
In order to attack this problem, we try to determine the set $RE_n^+(1,2)$ and $RE_n^-(0,3)$.
We note that $RE_n^+(1,2)$ and $RE_n^-(0,3)$ can be embedded into a complete lattice.
Indeed, both $RE_n^+(1,2)$ and $RE_n^-(0,3)$ are subsets of $RE_n$ where we set $RE_n=N_{n/n-1}^{-1}(\{\pm 1\})$ and $N_{n/n-1}:\Z[X_n]\to \Z[X_{n-1}]$ is the relative norm map (cf.~\cref{keytheorem_12} and \cref{keytheorem_03}).
Moreover we can regard $RE_n$ as a complete lattice in $\R^{2^{n-1}}$ via the map\footnote{Here we expand the domain of $\s$ to $\B_n$ by $\s(X_n)=\sqrt{2+\s(X_{n-1})}$ for each $\s \in \Gal(\B_{n-1}/\Q)$.}
\begin{equation*}
    l_n:RE_n\to \R^{2^{n-1}};\e \mapsto (\log|\s(\e)|)_{\s \in \Gal(\B_{n-1}/\Q)},
\end{equation*}
Thus the problem of determining $RE_n^+(1,2)$ and $RE_n^-(0,3)$ is reduced to the problem of determining certain points in the lattice $l_n(RE_n)$.
Since it seems also difficult, we consider the following problem.
\begin{problem}
    Investigating the distributions of $l_n(RE_n^+(1,2))$ and $l_n(RE_n^-(0,3))$ in $l_n(RE_n)$.
\end{problem}

In the case of $n=2$ and $(N,\ell)=(1,2)$, we plot $l_2(RE_2^+(1,2))$ in $l_2(RE_2)$ (\Cref{plot12}) .
Here the symbol $\bullet$ means that corresponding $\e$ is included in $RE_2^+(1,2)$, and the symbol of a circled $\bullet$ means that both $\e$ and $-\e$ are included in $RE_2^+(1,2)$.
\begin{figure}[h]
  \centering
  \begin{minipage}[b]{0.49\columnwidth}
      \centering
      \includegraphics[width=0.9\columnwidth]{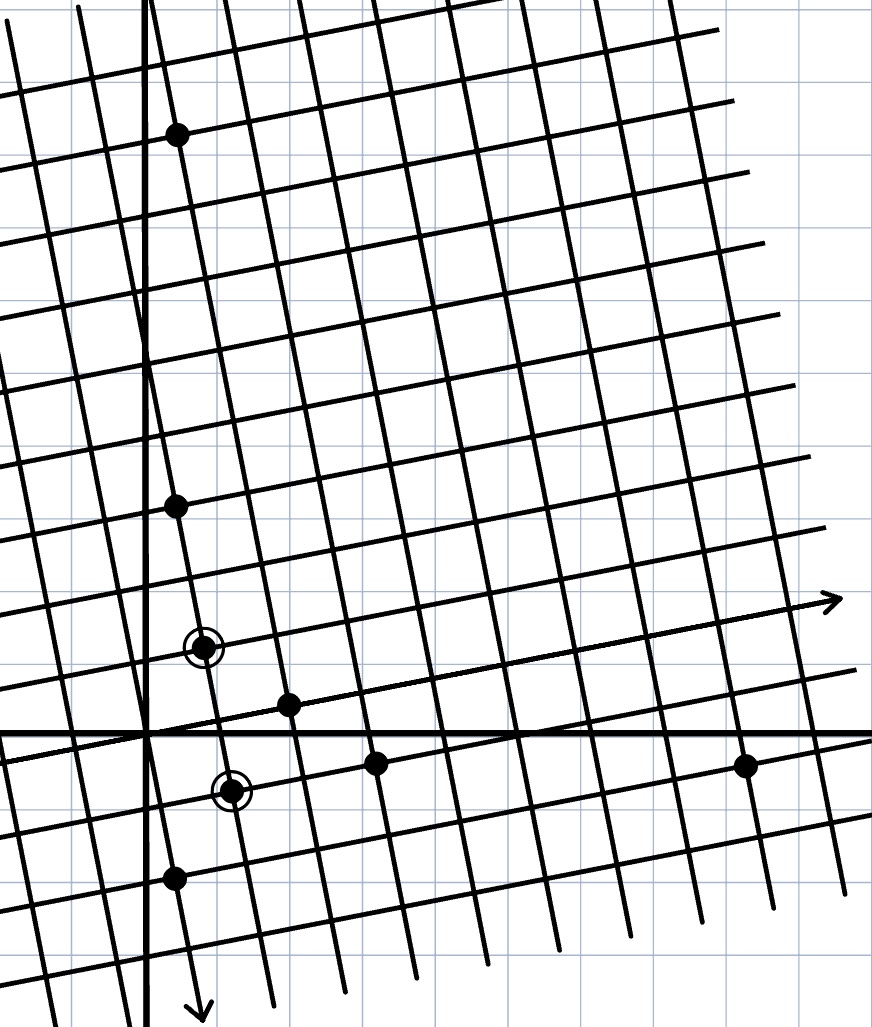}
      \caption{$l_2(RE_2^+(1,2))$}\label{plot12}
  \end{minipage}
  \begin{minipage}[b]{0.49\columnwidth}
        \centering
        \includegraphics[width=0.9\columnwidth]{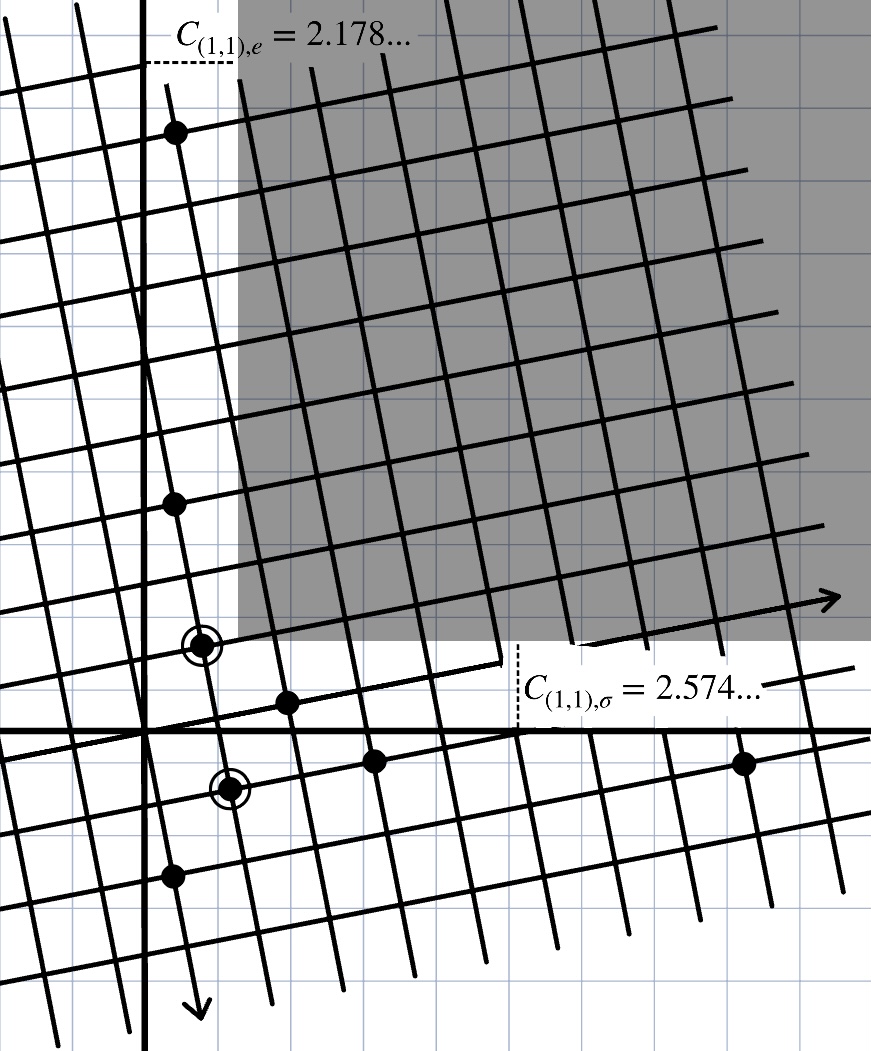}
        \caption{$l_2(RE_2^+(1,2))$}\label{REwithShade}
  \end{minipage}
\end{figure}
We see that the points of $l_2(RE_2^+(1,2))$ are located near the $x$ or $y$-axis in \Cref{plot12}.
We note that the similar phenomena occurs in the case of $(N,\ell)=(0,3)$.

This phenomena can be explained by our second main theorem below.
\begin{theorem}\label{maintheorem2} (\cref{maintheorem2_explicit_12} and \cref{maintheorem2_explicit_03})
    We fix $n\in \Z_{\geq 1}$.
    For each $\s \in \Gal(\B_{n-1}/\Q)$ and $s\in \{1, -1\}^{2^{n-1}}$ there exists a computable value $C_{s, \s}$ such that
    if $\e \in RE_n^+(1,2)$ then $\log|\s(\e)|<C_{s,\s}$ for some $\s \in \Gal(\B_{n-1}/\Q)$ and $s \in \{1, -1\}^{2^{n-1}}$.
    We also obtain a similar result for $(0,3)$-type.
\end{theorem}
For simplicity, we only describe the result of applying \cref{maintheorem2} to the first quadrant.
Then we obtain that there is no $l_2(RE_2^+(1,2))$ element in the shaded area in \Cref{REwithShade}.
Here $e$ in \Cref{REwithShade} denote the unit element of $\Gal(\B_{n-1}/\Q)$ and $\s$ in \Cref{REwithShade} denote the conjugation $2\cos(2\pi/2^3)\mapsto 2\cos(3\times2\pi/2^3)$.

In \cite{Yoshizaki2023} the second author pointed out that there is an interesting relationship between certain $\Z[X_{n-1}]$-PCF of type $(1,2)$ for $X_{n}$ and a generalization of the Pell equation (see \Cref{rel_to_pell_eq}).
In the last section, we find a similar relationship between the new $\Z[X_{n-1}]$-PCFs of type $(0,3)$ for $X_{n}$ given in \cref{maintheorem1} and the generalized Pell equation (\cref{fundsolofXn}).

The plan of this paper is as follows.
In \Cref{preliminaries}, we recall some notations and properties which we need in the proof of main theorems.
In \Cref{key_theorem}, we give key theorems to prove \cref{maintheorem1,maintheorem2}. 
In \Cref{main_thm_1} (resp.\ \ref{main_thm_2}), we state \cref{maintheorem1} (resp.\ \cref{maintheorem2}) explicitly and give a proof.
In \Cref{rel_to_pell_eq}, we describe the relationship between a $\Z[X_{n-1}]$-PCF of type $(0,3)$ for $X_{n}$ and the generalized Pell equation.

%% file: sec2_prelim.tex
In this section we prepare some notations and propeties.

\subsection{The $\Z_2$-extension over the rationals}

Let $\B_{\infty}$ be the $\Z_2$-extension over $\Q$, that is, the infinite Galois extension over $\Q$ whose Galois group is isomorphic to the group of $2$-adic integers $\Z_2$ as a topological group.
We set $\B_n=\Q(X_n)$.
Recall that $X_n = X_n = 2\cos(2\pi/2^{n+2})$.
For example,
\[
    X_{1} = \sqrt{2},\quad X_{2} = \sqrt{2+\sqrt{2}},\quad X_{3} = \sqrt{2+\sqrt{2+\sqrt{2}}}, \dots.
\]
Then we have $\G(\B_n/\Q)\cong\Z/2^n\Z$ and $\B_{\infty}=\cup_n \B_n$.
We set $\zeta_{2^n}=\exp(2\pi\sqrt{-1}/2^n)$.
Since $X_n=\zeta_{2^{n+2}}+\zeta_{2^{n+2}}^{-1}$, the ring of integers of $\B_n$ is $\Z[X_n]$ (see \cite[Proposition 2.16]{Washington}).
Let $N_n:\Z[X_n]\to \Z;x \mapsto \prod_{\s \in \G(\B_n/\Q)}\s(x)$ denote the norm map of $\B_n$.
Let $\tau_n$ denote the generator of $\G(\B_n/\B_{n-1})\cong \Z/2\Z$, and $N_{n/n-1}:\Z[X_n]\to\Z[X_{n-1}];x\mapsto x\tau_n(x)$ be the relative norm map.
We note that for every $x \in \Z[X_n]$ there exist unique elements $p, q \in \Z[X_{n-1}]$ such that $x=p+X_nq$, and we have $N_{n/n-1}(x)=p^2-X_n^2q^2$.
In this article we write $x=p(x)+X_nq(x)$.
We set
\begin{align*}
    RE_n^{-}&=\{\e \in \Z[X_n] \mid N_{n/n-1}(\e)=-1\}, \\
    RE_n^{+}&=\{\e \in \Z[X_n] \mid N_{n/n-1}(\e)=1\},
\end{align*}
and $RE_n = RE^{-}_n \cup RE^{+}_n$.
Note that $RE_n^{+}$ and $RE_n$ are subgroups of the unit group $\Z[X_n]^{\times}$.

\subsection{Periodic continued fraction over $\Z[X_{n-1}]$}

Suppose $c_i \in \Z[X_{n-1}]$ for $i\in \Z_{\geq 0}$.
In a similar manner to a regular continued fraction, 
we define a finite continued fraction $[c_0, \dots, c_m]$ and a $(N, \ell)$-type periodic continued fraction $[c_0,\dots, c_{N-1}, \overline{c_{N},\dots, c_{N+\ell-1}}]$.
The $k$-th convergent $p_k/q_k$ of a periodic continued fraction is defined by
\begin{align*}
    \begin{pmatrix}
        p_k & p_{k-1}\\
        q_k & q_{k-1}
    \end{pmatrix}
    = 
    \begin{pmatrix}
        c_0 & 1\\
        1 & 0
    \end{pmatrix}\begin{pmatrix}
        c_1 & 1\\
        1 & 0
    \end{pmatrix}\dots\begin{pmatrix}
        c_k & 1\\
        1 & 0
    \end{pmatrix}
\end{align*}
for $k\in \Z_{\geq 0}$.
Then we have
\begin{align}\label{property:determinant}
    p_{k}q_{k-1} - p_{k-1}q_{k} = (-1)^{k+1},
\end{align}
and
\begin{align}\label{property:convergent}
    \frac{p_k}{q_k} = [c_0, \dots, c_k]
\end{align}
for $k\in\Z_{\geq 1}$ (cf.\ \cite[p.385 (8) and p.386 (9)]{BrockElkiesJordan2021}).

For a matrix $A = \displaystyle\begin{pmatrix}a & b\\ c & d\end{pmatrix} \in GL_2(\C)$,
we regard $A$ as the automorphism of the projective line $\PP^1$ over $\C$ by
\[
    z \mapsto \frac{az+b}{cz+d}
\]
for $z\in\PP^1(\C)$.\footnote{Recall that we define $\frac{a\infty+b}{c\infty+d} = \frac{a}{c}$.}
By using this automorphism, 
we regard a finite continued fraction $[c_0, \dots, c_m]$ as
\[
    \begin{pmatrix}
        c_0 & 1\\
        1 & 0
    \end{pmatrix}\begin{pmatrix}
        c_1 & 1\\
        1 & 0
    \end{pmatrix}\dots\begin{pmatrix}
        c_m & 1\\
        1 & 0
    \end{pmatrix}\infty.
\]

\begin{lemma}\label{lemmaforpelleq}
    Suppose $\alpha\in\Z[X_{n}]\setminus\Z[X_{n-1}]$ and let
    \[
        X^2+bX+c\in\Z[X_{n-1}][X].
    \]
    be the minimal polynomial of $\alpha$.
    If $A\in GL_{2}(\Z[X_{n-1}])$ satisfies $A \alpha = \alpha$,
    there exist $x, y \in \Z[X_{n-1}]$ such that
    \begin{align*}
        A = \begin{pmatrix}
            \frac{x-by}{2} & -cy\\
            y & \frac{x+by}{2}
        \end{pmatrix}.
    \end{align*}
\end{lemma}

\begin{proof}
    Suppose $A = \displaystyle\begin{pmatrix}s & t\\ u & v\end{pmatrix} \in \mathrm{GL}_2(\Z[X_{n-1}])$ satisfies $A\alpha = \alpha$.
    Then, we obtain
    \begin{align*}
        &\frac{s\alpha + t}{u\alpha + v} = \alpha,\\
        &s\alpha + t = u\alpha^2 + v\alpha,\\
        &u\alpha^2 + (v-s)\alpha - t = 0,
    \end{align*}
    and $\alpha$ is a root of the polynomial $uX^2 + (v-s)X - t = 0\in \Z[X_{n-1}][X]$. 
    Since $\alpha\in\Z[X_{n}]\setminus\Z[X_{n-1}]$,
    there exists $y\in\Z[X_{n-1}][X]$ such that
    \begin{align*}
        uX^2 + (v-s)X - t = y(X^2 + bX + c),
    \end{align*}
    and we obtain
    \begin{align*}
        u = y, \quad v-s = by, \quad \text{and} \quad -t=cy.
    \end{align*}
    Setting $x = 2v - bu \in \Z[X_{n-1}]$,
    we have
    \begin{align*}
        A = \begin{pmatrix} 
            \frac{x-by}{2} & -cy\\
            y & \frac{x+by}{2}
        \end{pmatrix}.
    \end{align*}
\end{proof}

\begin{proposition}\label{1elltypetoPelleq}
    Let
    \begin{equation}\label{PCFrepresentaiton}
        X_{n}=[a_0,\overline{a_1,...,a_{\ell}}]
    \end{equation}
    be a $(1,\ell)$-type PCF over $\Z[X_{n-1}]$.
    Then we have
    \begin{equation*}
        p_{\ell-1}^2-X_{n}^2q_{\ell-1}^2=(-1)^{\ell}.
    \end{equation*}
\end{proposition}

\begin{proof}
    We have
    \begin{align*}
        X_{n}
        &= [a_{0}, \overline{a_1, \dots, a_{\ell}}]\\
        &= \left[a_{0}, a_{1}, \dots, a_{\ell}, \frac{1}{X_{n}-a_{0}}\right]\\
        &= \begin{pmatrix}
            a_0 & 1\\
            1 & 0
        \end{pmatrix}\begin{pmatrix}
            a_1 & 1\\
            1 & 0
        \end{pmatrix} \dots \begin{pmatrix}
            a_{\ell} & 1\\
            1 & 0
        \end{pmatrix}\begin{pmatrix}
            0 & 1\\
            1 & -a_{0}
        \end{pmatrix}X_{n}\\
        &= \begin{pmatrix}
            p_{\ell} & p_{\ell-1}\\
            q_{\ell} & q_{\ell-1}
        \end{pmatrix}\begin{pmatrix}
            0 & 1\\
            1 & -a_{0}
        \end{pmatrix}X_{n}.
    \end{align*}
    Then we obtain
    \begin{align}\label{matrixequation}
        \begin{pmatrix}
            p_{\ell} & p_{\ell-1}\\
            q_{\ell} & q_{\ell-1}
        \end{pmatrix}\begin{pmatrix}
            0 & 1\\
            1 & -a_{0}
        \end{pmatrix} = \begin{pmatrix}
            \frac{x}{2} & X_{n}^2y\\
            y & \frac{x}{2}
        \end{pmatrix}
    \end{align}
    for some $x,y \in \Z[X_{n-1}]$ by \cref{lemmaforpelleq}.
    Hence, we obtain $x=2p_{\ell-1}$ and $y=q_{\ell-1}$.
    By taking the determinant of both side of (\ref{matrixequation}),
    we have
    \begin{align*}
        -(p_{\ell}q_{\ell-1} - p_{\ell-1}q_{\ell})
        = \left(\frac{x}{2}\right)^2 - X_{n}^2y^2
        = p_{\ell-1}^2 - X_{n}^2q_{\ell-1}^2.
    \end{align*}
    Combining with (\ref{property:determinant}), we obtain the desired result.
\end{proof}

In a similar manner to the proof of \cref{1elltypetoPelleq}, 
we can show the following.
\begin{proposition}\label{0elltypetoPelleq}
    Let
    \begin{align*}
        X_{n}=[\overline{a_0,...,a_{\ell-1}}]
    \end{align*}
    be a $(0,\ell)$-type PCF over $\Z[X_{n-1}]$. 
    Then we have 
    \begin{equation*}
        p_{\ell-1}^2-X_{n}^2q_{\ell-1}^2=(-1)^{\ell}.
    \end{equation*}
\end{proposition}


\subsection{PCF varieties}

In this subsection,
we recall the definition of a PCF variety over $\Z[X_{n-1}]$.
For a finite continued fraction $[c_0, \dots, c_m]$ over $\Z[X_{n-1}]$,
define
\begin{align*}
M([c_0,c_1\dots,c_m]) 
&= \begin{pmatrix}
M([c_0,c_1\dots,c_m])_{11} & M([c_0,c_1\dots,c_m])_{12} \\
M([c_0,c_1\dots,c_m])_{12} & M([c_0,c_1\dots,c_m])_{22} \\
\end{pmatrix}\\
&= 
\begin{pmatrix}
    c_0 & 1 \\
    1 & 0 
\end{pmatrix}
\begin{pmatrix}
    c_1 & 1 \\
    1 & 0 
\end{pmatrix}\dots
\begin{pmatrix}
    c_m & 1 \\
    1 & 0 
\end{pmatrix}.
\end{align*}
For $N\in\Z_{\geq 0}$ and $\ell\in\Z_{\geq 1}$, 
we set 
\begin{align*}
&E(y_1,\dots, y_N, x_1,\dots, x_{\ell})\\
&=
\begin{pmatrix}
E(y_1,\dots,y_N,x_1,\dots,x_{\ell})_{11} & E(y_1,\dots,y_N,x_1,\dots,x_{\ell})_{12} \\
E(y_1,\dots,y_N,x_1,\dots,x_{\ell})_{21} & E(y_1,\dots,y_N,x_1,\dots,x_{\ell})_{22} \\
\end{pmatrix}\\
&:=
\begin{cases}
M([y_1,\dots,y_N,x_1,\dots,x_{\ell},0,-y_N,\dots,-y_1,0]) & \text{if $N\geq 1$,}\\
M([x_1,\dots,x_{\ell}]) & \text{if $N=0$}. 
\end{cases}
\end{align*}
Note that $E(y_1,\dots, y_N, x_1,\dots, x_{\ell}) = E(x_1,\dots, x_{\ell})$ when $N=0$.

\begin{definition}
Fix $N\in\Z_{\geq 0}$ and $\ell \in \Z_{\geq 1}$.
We define a PCF variety $V(X_{n})_{N,\ell}$ of $(N, \ell)$-type by 
\begin{align*}
\begin{cases}
E(y_1,\dots,y_N,x_1,\dots,x_{\ell})_{22}-E(y_1,\dots,y_N,x_1,\dots,x_{\ell})_{11}=0,\\
E(y_1,\dots,y_N,x_1,\dots,x_{\ell})_{12}=X_{n}^2 E(y_1,\dots,y_N,x_1,\dots,x_{\ell})_{21},
\end{cases}
\end{align*}
where $y_1,\dots,y_N, x_1, \dots, x_{\ell}$ are variables.
In what follows, $(y_1,\dots,y_N,x_1,\dots,x_{\ell})$ denotes the coordinate of $V(X_{n})_{N, \ell}$.
\end{definition}

In this paper, 
we deal with the cases $(N, \ell) = (0,3)$ and $(1,2)$, that is,
\begin{align*}
    V(X_{n})_{1,2}\colon\begin{cases}
        x_1 - (y_1(x_1x_2+1) + x_2) = 0,\\
        x_1y_1+1 = X_{n}^2(x_1x_2+1),
    \end{cases}
\end{align*}
and
\begin{align*}
    V(X_{n})_{0,3}\colon\begin{cases}
        x_2 - (x_1(x_2x_3)+1 + x_3) = 0,\\
        x_1x_2+1 = X_{n}^2(x_2x_3 + 1).
    \end{cases}
\end{align*}

%% file: sec3_key_thm.tex
In this section, we give relationships between units in $\Z[X_{n}]$ and $(N,\ell)$-type PCFs over $\Z[X_{n-1}]$ for $X_{n}$ when $(N,\ell)=(1,2)$ and $(0,3)$.

\subsection{$(1,2)$-type}

\begin{theorem}\label{keytheorem_12}
    We fix $n\in\Z_{\geq 1}$.
    Let $X_{n}=[a_0,\overline{a_1,a_2}]$ be a $(1,2)$-type continued fraction over $\Z[X_{n-1}]$.
    Then there exists $\e \in RE_{n}^+$ such that $a_0=(p(\e)-1)/q(\e)$, $a_1=q(\e)$, $a_2=2(p(\e)-1)/q(\e)$, and $\sgn(p(\e)q(\e))=1$.
    Conversely, if $\e \in RE_{n}^+$ satisfies $q(\e) \mid p(\e)-1$ and $\sgn(p(\e)q(\e))=1$ then we have
    \begin{align*}
        X_{n}=\left[\frac{p(\e)-1}{q(\e)},\overline{q(\e), 2\frac{p(\e)-1}{q(\e)}}\right].
    \end{align*}
\end{theorem}

In other words, this theorem states that there is a bijection
\[
    P_{n}^{1,2}:RE_{n}^{+}(1,2) \to \left\{[a_0, \overline{a_1,a_2}] \setmid X_{n} = [a_0,\overline{a_1,a_2}]\right\},
\]
where $RE_n^+(1,2)=\{\e \in RE_n^+\mid q(\e)\mid p(\e)-1 \text{ and }\sgn(p(\e)q(\e))=1\}$.
\begin{proof}
    First we suppose that there exists $\e \in RE_{n}^+$ such that $q(\e)\mid p(\e)-1$.
    Here we note that $\e \neq \pm1$.
    If we put $x_0=(p(\e)-1)/q(\e)$, $x_1=q(\e)$, and $x_2=2(p(\e)-1)/q(\e)$, then we see that $(x_0,x_1,x_2)$ satisfies the equation (\ref{1_2_def_poly}).
    Thus we have $(x_0,x_1,x_2)\in V(X_{n})_{1,2}(\Z[X_{n-1}])$.
    By using the Brock--Elkies--Jordan's convergence criteria \cite[Theorem 4.3]{BrockElkiesJordan2021}, we check the convergence of $[x_0,\overline{x_1,x_2}]$.
    Firstly, we have
    \begin{align*}
        E(x_0,x_1,x_2)=
        \begin{pmatrix}
            p(\e) & X_{n}^2q(\e) \\
            q(\e) & p(\e)
        \end{pmatrix}
        \neq \sqrt{-1}^2I.
    \end{align*}
    Secondly, we have
    \begin{align*}
        M([x_1,x_2])_{21}=M([x_2,x_1])_{21}=x_2\neq0.
    \end{align*}
    Finally, we have
    \begin{align*}
        (-1)^2\Tr(E(x_0,x_1,x_2))^2=4p(\e)^2\geq 4
    \end{align*}
    The last inequality is due to $p(\e)^2=1+X_{n}^2q(\e)^2\geq 1$.
    Hence we complete the proof of the convergence.
    We determine the value of $[x_0,\overline{x_1,x_2}]$.
    The eigenvalue of $E(x_0,x_1,x_2)$ whose absolute value is greater than or equal to $1$ is $\lambda=p(\e)+\sgn(p(\e))|X_{n}q(\e)|$.
    Thus the value of $[x_0,\overline{x_1,x_2}]$ is
    \begin{align}\label{signOfPCF12}
        [x_0,\overline{x_1,x_2}]&=\frac{\lambda -p(\e)}{q(\e)}=\sgn(p(\e)q(\e))X_{n}.
    \end{align}

    Next we suppose that $X_{n}=[a_0,\overline{a_1,a_2}]$ for some $a_0,a_1$, and $a_2\in\Z[X_{n-1}]$.
    Since $a_2=2a_0$,
    it holds that
    \begin{align}\label{1_2_def_poly}
        a_0^2a_1+2a_0=X_{n}^2a_1.
    \end{align}
    From \cref{1elltypetoPelleq}, we see that 
    \[
        p_1^2-X_{n}^2q_1^2=(a_0a_1+1)^2-X_{n}^2a_1^2=1.
    \]
    Thus we have $a_0a_1+1+X_{n}a_1 \in RE_{n}^+$.
    By putting $\e=a_0a_1+1+X_{n}a_1$, we obtain that $a_0=(p(\e)-1)/q(\e)$, $a_1=q(\e)$, and $a_2=2(p(\e)-1)/q(\e)$.
    In a similar manner to (\ref{signOfPCF12}),
    we can show that the signature of a PCF $[(p(\e)-1)/q(\e),\overline{q(\e),2(p(\e)-1)/q(\e)}]$ coincides with the signature of $p(\e)q(\e)$, that is, $\sgn(p(\e)q(\e))=1$.
\end{proof}

\subsection{$(0,3)$-type}

Before we state the theorem,
we show the following lemma.
\begin{lemma}\label{DefPoly}
    If $a_1,a_2,a_3 \in \Z[X_{n-1}]$ satisfy
    \begin{numcases}{}
        a_{1}a_{2} + 1 = X_{n}^{2}(a_{2}a_{3} + 1), \label{eq_1}\\
        a_{2}^{2} - X_{n}^{2}(a_{2}a_{3} + 1)^{2} = -1, \label{eq_2}
    \end{numcases}
    then we have $a_1a_2a_3+a_1+a_3=a_2$.
\end{lemma}

\begin{proof}
    Since $a_2\neq0$, we have $a_1=(X_{n}^2(a_2a_3+1)-1)/a_2$ by (\ref{eq_1}).
    Thus we have
    \begin{align*}
        a_1a_2a_3+a_1+a_3-a_2&=(X_{n}^2(a_2a_3+1)-1)a_3+\frac{X_{n}^2(a_2a_3+1)-1}{a_2}+a_3-a_2 \\
        &=\frac{1}{a_2}(X_{n}^2(a_2a_3+1)^2-a_2^2-1).
    \end{align*}
    Since $X_{n}^2(a_2a_3+1)^2-a_2^2-1=0$ by (\ref{eq_2}), we have $a_1a_2a_3+a_1+a_3=a_2$.
\end{proof}

\begin{theorem}\label{keytheorem_03}
    We fix $n\in\Z_{\geq 1}$.
    Let $X_{n}=[\overline{a_1,a_2,a_3}]$ be a $(0,3)$-type continued fraction over $\Z[X_{n-1}]$.
    Then there exists $\e \in RE_{n}^-$ such that $a_1=(X_{n}^2q(\e)-1)/p(\e)$, $a_2=p(\e)$, $a_3=(q(\e)-1)/p(\e)$, and $\sgn(p(\e)q(\e))=1$.
    Conversely, if $\e \in RE_{n}^-$ satisfies $p(\e) \mid X_{n}^2q(\e)-1$, $p(\e)\mid q(\e)-1$, and $p(\e)q(\e)=1$, then we have
    \begin{align*}
        X_{n}=\left[\overline{\frac{X_{n}^2q(\e)-1}{p(\e)}, p(\e), \frac{q(\e)-1}{p(\e)}}\right].
    \end{align*}
\end{theorem}

In other words, this theorem states that there is a bijection
\[
    P_{n}^{0,3}:RE_{n}^{-}(0,3) \to \left\{[\overline{a_1,a_2,a_3}] \setmid X_{n} = [\overline{a_1,a_2,a_3}]\right\},
\]
where 
\[
    RE_n^-(0,3)=\{\e \in RE_n^-\mid p(\e)\mid q(\e)-1,\ p(\e)\mid X_n^2q(\e)-1\text{, and }\sgn(p(\e)q(\e))=1\}.
\]

\begin{proof}
    First we suppose that there exists $\e \in RE_{n}^-$ such that $p(\e) \mid X_{n}^2q(\e)-1$ and $p(\e)\mid q(\e)-1$.
    If we put $x_1=(X_{n}^2q(\e)-1)/p(\e)$, $x_2=p(\e)$, and $x_3=(q(\e)-1)/p(\e)$, then we see that $(x_1,x_2,x_3)$ satisfies the equation (\ref{0_3_def_poly_2}), and also satisfies (\ref{0_3_def_poly_1}) by \cref{DefPoly}.
    Thus we have $(x_1,x_2,x_3)\in V(X_{n})_{0,3}(\Z[X_{n-1}])$.
    By using the Brock--Elkies--Jordan's convergence criteria \cite[Theorem 4.3]{BrockElkiesJordan2021}, we check the convergence of $[\overline{x_1,x_2,x_3}]$.
    Firstly, we have
    \begin{align*}
        E(x_1,x_2,x_3)=
        \begin{pmatrix}
            p(\e) & X_{n}^2q(\e) \\
            q(\e) & p(\e)
        \end{pmatrix}
        \neq \sqrt{-1}^3I.
    \end{align*}
    Secondly, we have
    \begin{align*}
        M([x_1,x_2,x_3])_{21}&=x_2x_3+1=q(\e)\neq0, \\        
        M([x_3,x_1,x_2])_{21}&=x_1x_2+1=X_{n}^2q(\e)\neq0, \\
        M([x_2,x_3,x_1])_{21}&=x_3x_1+1 \\
        &=\frac{X_{n}^2q(\e)^2-X_{n}^2q(\e)-q(\e)+1+p(\e)^2}{p(\e)^2} \\
        &=\frac{q(\e)}{p(\e)^2}((2q(\e)-1)X_{n}^2-1)\neq0.
    \end{align*}
    The last inequality is due to the fact that $X_{n}^2$ is not a unit in $\Z[X_{n-1}]$.
    Finally, we have
    \begin{align*}
        (-1)^3\Tr(E(x_1,x_2,x_3))^2=-4p(\e)^2<0.
    \end{align*}
    Hence we complete the proof of the convergence.
    We determine the value of $[x_1,x_2,x_3]$.
    The eigenvalue of $E(x_1,x_2,x_3)$ whose absolute value is greater than or equal to $1$ is $\lambda=p(\e)+\sgn(p(\e))|X_{n}q(\e)|$.
    Thus the value of $[x_1,x_2,x_3]$ is
    \begin{align}\label{signOfPCF03}
        [\overline{x_1,x_2,x_3}]&=\frac{\lambda -p(\e)}{q(\e)}=\sgn(p(\e)q(\e))X_{n}.
    \end{align}

    Next we suppose that $X_{n}=[\overline{a_1,a_2,a_3}]$ for some $a_1,a_2$, and $a_3\in\Z[X_{n-1}]$.
    Then we obtain
    \begin{align}
        a_1a_2a_3+a_1+a_3 &= a_2, \label{0_3_def_poly_1}\\
        a_1a_2+1 &= X_{n}^2(a_2a_3+1), \label{0_3_def_poly_2}
    \end{align}
    from the defining equation of $V(X_{n})_{0,3}$.
    From \cref{0elltypetoPelleq}, we see that $p_{2}^2-X_{n}^2q_{2}^2=(a_1a_2a_3+a_1+a_3)^2-X_{n}^2(a_2a_3+1)^2=-1$.
    Thus, by using (\ref{0_3_def_poly_1}), we have $a_2+X_{n}(a_2a_3+1) \in RE_{n}^-$.
    By putting $\e=a_2+X_{n}(a_2a_3+1)$, we obtain that $a_1=(X_{n}^2q(\e)-1)/p(\e)$, $a_2=p(\e)$, and $a_3=(q(\e)-1)/p(\e)$.
    In a similar manner to (\ref{signOfPCF03}), 
    we can show that the signature of a PCF $[\overline{(X_{n}^2q(\e)-1)/p(\e),p(\e),(q(\e)-1)/p(\e)}]$ coincides with the signature of $p(\e)q(\e)$.
    Thus we obtain $\sgn(p(\e)q(\e))=1$.
\end{proof}

%% file: sec4_main_thm_1.tex
In this section, we state the explicit form of \cref{maintheorem1} and give its proof.
We recall the property of special units
\begin{align*}
  \delta_{n} = \frac{X_n + 2}{X_n - 2}
\end{align*}
and
\begin{align*}
  \eta_{n} = 1 + \sum_{k=1}^{2^n - 1}2\cos\left(\frac{k\pi}{2^{n+1}}\right).
\end{align*}

\begin{proposition}[{cf. \cite[Section 3 and (6.1)]{MorisawaOkazaki2020}}]\label{propertyOfUnits}
  We have $\delta_n \in RE_n^+$ and $\eta_n \in RE_n^-$ for all $n\in\Z_{\geq 1}$.
\end{proposition}

\begin{proof}
  For the first assertion, it is easy to check that $\delta_{n}\tau_n(\delta_{n})=1$.
  Thus we have to check that $\delta_n \in \Z[X_n]$.
  Since $N_{n}(X_n+2)=N_n(X_n-2)=2$ and $2$ is a totally ramified prime in $\B_n$, 
  both $X_n+2$ and $X_n-2$ are prime elements over $2$.
  Therefore we have $\delta_n \in \Z[X_n]$ and the first assertion follows.

  For the second assertion, we have
  \begin{align*}
      N_{n/n-1}(\eta_n)=\left(1+\sum_{k=1}^{2^{n-1}-1}2\cos\left(\frac{(2k)\pi}{2^{n+1}}\right)\right)^2-\left(\sum_{k=0}^{2^{n-1}-1}2\cos\left(\frac{(2k+1)\pi}{2^{n+1}}\right)\right)^2.
  \end{align*}
  by $\tau_n(2\cos((k\pi)/2^{n+1}))=(-1)^k2\cos((k\pi)/2^{n+1})$.
  Since $2\cos((k\pi)/2^{n+1})=\zeta_{2^{n+2}}^k+\zeta_{2^{n+2}}^{-k}$,
  we have
  \begin{align*}
    &N_{n/n-1}(\eta_n) \\
    &=1+\sum_{k=1}^{2^{n-1}-1}2(\zeta_{n+2}^{2k}+\zeta_{n+2}^{-2k})+\sum_{k=1}^{2^{n-1}-1}(\zeta_{n+2}^{4k}+2+\zeta_{n+2}^{-4k}) \\
    &+\sum_{\substack{1\leq k,l \leq 2^{n-1}-1 \\ k\neq l}}2(\zeta_{2^{n+2}}^{2k+2l}+\zeta_{2^{n+2}}^{-2k-2l}+\zeta_{2^{n+2}}^{2k-2l}+\zeta_{2^{n+2}}^{2l-2k}) \\
    &-\sum_{k=0}^{2^{n-1}-1}(\zeta_{n+2}^{2(2k+1)}+2+\zeta_{n+2}^{-2(2k+1)})-\sum_{\substack{0\leq k,l \leq 2^{n-1}-1 \\ k\neq l}}2(\zeta_{2^{n+2}}^{2k+2l+2}+\zeta_{2^{n+2}}^{-2k-2l-2}+\zeta_{2^{n+2}}^{2k-2l}+\zeta_{2^{n+2}}^{2l-2k}) \\
    &=1+\sum_{k=1}^{2^{n-1}-1}(\zeta_{2^{n+1}}^{2k}+\zeta_{2^{n+1}}^{-2k})+\sum_{k=0}^{2^{n-1}-1}(\zeta_{2^{n+1}}^{2k+1}+\zeta_{2^{n+1}}^{-2k-1})-2 \\
    &=-1+\sum_{i=2}^{n+1}\Tr_{\Q(\zeta_{2^{i}})/\Q}(\zeta_{2^{i}}).
  \end{align*}
  Combining with $\Tr_{\Q(\zeta_{2^{i}})/\Q}(\zeta_{2^{i}})=0$ for all $i\geq 2$, 
  the second assertion follows.
\end{proof}
The following is the explicit form of \cref{maintheorem1}.
\begin{theorem}\label{maintheorem1_explicit}
  For every $n\in \Z_{\geq1}$ we have
  \begin{align*}
      X_{n}&=\left[2,\overline{\frac{4}{X_{n-1}-2},4}\right] \\
      &=\left[\overline{\frac{(\eta_{n}-\eta_{n-1})X_{n}-1}{\eta_{n-1}},\eta_{n-1},\frac{(\eta_{n}-\eta_{n-1})/X_{n}-1}{\eta_{n-1}}}\right].
  \end{align*}
\end{theorem}
\begin{proof}
  We show that $\delta_{n}$ (resp.\ $\eta_{n}$) satisfies the condition of \cref{keytheorem_12} (resp.\ \cref{keytheorem_03}).
  Recall that $p(a)=(a+\tau_{n}(a))/2$ and $q(a)=(a-\tau_{n}(a))/(2X_{n})$ for all $a\in \Z[X_{n}]$.

  By \cref{propertyOfUnits},
  we have $\delta_{n}\in RE_{n}^+$.
  Moreover, we obtain
  \begin{align*}
    p(\delta_{n})&=\left(\frac{X_{n}+2}{X_{n}-2}+\frac{X_{n}-2}{X_{n}+2}\right)\cdot \frac{1}{2} \\
    &=\frac{2X_{n}^2+8}{X_{n}^2-4}\cdot\frac{1}{2} \\
    &=\frac{X_{n}^2+4}{X_{n}^2-4} \\
    &=\frac{X_{n-1}+6}{X_{n-1}-2}
  \end{align*}
  and
  \begin{align*}
    q(\delta_{n})&=\left(\frac{X_{n}+2}{X_{n}-2}-\frac{X_{n}-2}{X_{n}+2}\right)\cdot \frac{1}{2X_{n}} \\
    &=\frac{8X_{n}}{X_{n}^2-4}\cdot\frac{1}{2X_{n}} \\
    &=\frac{4}{X_{n-1}-2}.
  \end{align*}
  Thus we see that $q(\delta_{n})\mid p(\delta_{n})-1$. 
  Since $\sgn(p(\delta_{n})q(\delta_{n}))=1$ and $(p(\delta_{n})-1)/q(\delta_{n})=2$, 
  the former assertion follows.

  By \cref{propertyOfUnits}, 
  we have $\eta_{n}\in RE_{n}^{-}$.
  Moreover, we obtain
  \begin{align*}
    p(\eta_{n})=\eta_{n-1}
  \end{align*}
  by the definition of $\eta_{n}$.
  Since $\eta_{n-1}$ is a unit in $\Z[X_{n-1}]$, we see that $p(\eta_{n})\mid q(\eta_{n})-1$ and $p(\eta_{n})\mid X_{n}^2q(\eta_{n})-1$.
  Since $\eta_{n}=\eta_{n-1}+X_{n}q(\eta_{n})$ we have
  \begin{align*}
    q(\eta_{n})=\frac{\eta_{n}-\eta_{n-1}}{X_{n}}
  \end{align*}
  and $\sgn(p(\eta_{n})q(\eta_{n}))=1$.
  Hence the later assertion follows.
\end{proof}

%% file: sec5_main_thm_2.tex
This section gives the proof of \cref{maintheorem2}.
We define an injective map $\mu_n:\B_{n-1}\to \R^{2^{n-1}}; x \mapsto (\s(x))_{\s\in \Gal(\B_{n-1}/\Q)}$.
Then $\mu_n(\Z[X_{n-1}])$ is a complete lattice in $\R^{2^{n-1}}$.
We expand the domain of $\s$ to $\B_n$ by $\s(X_n)=\sqrt{2+\s(X_{n-1})}$ for each $\s \in \Gal(\B_{n-1}/\Q)$.
For each $s=(s_{\s})_{\s \in \Gal(\B_{n-1}/\Q)} \in \{1, -1\}^{2^{n-1}}$ we define
\begin{equation*}
    \mu_{s}(X_n)=(s_{\s}\s(X_n))_{\s \in \Gal(\B_{n-1}/\Q)}.
\end{equation*}
Let $||\cdot||$ denote the Euclidean distance in $\R^{2^{n-1}}$.
Let $a_s$ be an element in $\Z[X_{n-1}]$ such that $||\mu_{s}(X_n)-\mu_n(a_{s})||=\min\{||\mu_{s}(X_n)-\mu_n(a)||\mid a\in\Z[X_{n-1}]\}$.
For $s=(1,...,1)$ we have $a_{s}=1$ for all $n$ by \cite[Remark 3.2 and Proposition 3.3]{Yoshizaki2023}.

\subsection{$(1,2)$-type}
We recall 
\[
  RE_n^+(1,2)=\{\e\in RE_n^+\mid q(\e)\mid p(\e) -1\text{ and }\sgn(p(\e)q(\e))=1\}.
\]
Then, by \cref{keytheorem_12}, there is a bijection between $RE_n^+(1,2)$ and the set of $(1,2)$-type PCFs of $X_n$.
\begin{theorem}\label{maintheorem2_explicit_12}
    We fix $n\in\Z_{\geq 1}$.
    We set $b_{s,\s}=2\s(X_{n})/|s_{\s}\s(X_{n})-\s(a_s)|$ and $C_{s,\s}=\log(\sqrt{b_{s,\s}^2+1}+|b_{s,\s}|)$ for each $s=(s_{\s})_{\s\in\Gal(\B_{n-1}/\Q)}\in\{1,-1\}^{2^{n-1}}$ and $\s \in \Gal(\B_{n-1}/\Q)$.
    If $\e\in RE_{n}^+(1,2)$, there exists $\s\in \Gal(\B_{n-1}/\Q)$ such that $|\log|\s(\e)||\leq C_{s,\s}$,
    where $s=(\sgn(\s(p(\e)q(\e))))_{\s\in\Gal(\B_{n-1}/\Q)}$.
\end{theorem}
\begin{proof}
  We prove the contraposition.
  We take $\e \in RE_{n}^{+}$ and omit ``$(\e)$'' from $p(\e)$ and $q(\e)$.
  We set $s=(\sgn(\s(p(\e)q(\e)))){\s\in\Gal(\B_{n-1}/\Q)}\in \{1,-1\}^{2^{n-1}}$.
  Then we have
  \begin{align*}
    \begin{cases}
      |\s(p)-\s(X_{n})\s(q)|<e^{-C_{s,\s}} & \text{if $|\s(\e)|>1$,}\\
      |\s(p)+\s(X_{n})\s(q)|<e^{-C_{s,\s}} & \text{if $|\s(\e)|<1$.}
    \end{cases}
  \end{align*}
  In either case, we have
  \begin{align*}
    |\s(p)|-|\s(X_{n})\s(q)|<e^{-C_{s,\s}}.
  \end{align*}
  Since $|\s(p)|=\sqrt{\s(X_{n})^2\s(q)^2+1}$, we have
  \begin{align*}
    \sqrt{\s(X_{n})^2\s(q)^2+1}<e^{-C_{s,\s}}+|\s(X_{n})\s(q)|
  \end{align*}
  and
  \begin{align}\label{q_bounds}
    |\s(q)|>\frac{e^{C_{s,\s}}-e^{-C_{s,\s}}}{2\s(X_{n})}.
  \end{align}
  An easy calculation shows that
  \begin{equation}\label{12mainBound}
      \frac{4\s(X_{n})}{e^{C_{s,\s}}-e^{-C_{s,\s}}}=|s_{\s}\s(X_{n})-\s(a_{s})|.
  \end{equation}
  Since $|\s(p)-\sgn(\s(pq))\s(X_{n})\s(q)|<1$, we have $|\s(p/q)-\sgn(\s(pq))\s(X_{n})|<|1/\s(q)|$.
  Thus we obtain
  \begin{align*}
    &\left|\s\left(\frac{p-1}{q}\right)-\sgn\left(\s\left(pq\right)\right)\s(X_{n})\right| \\
    &\leq \left|\s\left(\frac{p}{q}\right)-\sgn\left(\s\left(pq\right)\right)\s(X_{n})\right|+\left|\s\left(\frac{1}{q}\right)\right| \\
    &< \left| \s\left(\frac{2}{q}\right) \right|.
  \end{align*}
  Using (\ref{q_bounds}) and (\ref{12mainBound}) we have
  \begin{align*}
    \left|\s\left(\frac{p-1}{q}\right)-\sgn\left(\s\left(pq\right)\right)\s(X_{n})\right|<|s_{\s}\s(X_{n})-\s(a_s)|.
  \end{align*}
  Here we note that the nearest point of $\mu_{s}(X_{n})$ from $\mu_{n}(\Z[X_{n-1}])$ is $\mu_{n}(a_{s})$.
  Thus $\mu_{n}((p-1)/q)$ cannot be on the lattice $\mu_{n}(\Z[X_{n-1}])$, that is, $q\nmid p-1$ and $\e \not\in RE^{+}_{n}(1,2)$.
\end{proof}
\subsection{$(0,3)$-type}

Recall that
\[
  RE_n^-(0,3)=\{\e\in RE_n^-\mid p(\e)\mid q(\e)-1,\ p(\e)\mid X_n^2q(\e) -1\text{, and }\sgn(p(\e)q(\e))=1\}.
\]
Then, by \cref{keytheorem_03}, there is a bijection between $RE_{n}^-(0,3)$ and the set of $(0,3)$-type PCFs of $X_{n}$.
\begin{theorem}\label{maintheorem2_explicit_03}
  We fix $n\in\Z_{\geq 1}$.
  We set $b_{s,\s}=|\s(X_{n})+1|/|s_{\s}\s(X_n)-\s(a_s)|$ and $C_{s,\s}=\log(\sqrt{b_{s,\s}^2+1}+|b_{s,\s}|)$ for each $s\in\{1,-1\}^{2^{n-1}}$ and $\s \in \Gal(\B_{n-1}/\Q)$.
  If $\e\in RE_{n}^-(0,3)$, there exists $\s\in \Gal(\B_{n-1}/\Q)$ such that $|\log|\s(\e)||\leq C_{s,\s}$,
  where $s=(\sgn(\s(p(\e)q(\e))))_{\s\in\Gal(\B_{n-1}/\Q)}$.
\end{theorem}
\begin{proof}
  We prove the contraposition.
  In a similar manner to the proof of \cref{maintheorem2_explicit_12}, 
  we have
    \begin{equation*}
        |\s(X_{n})\s(q)|-|\s(p)|<e^{-C_{s,\s}}.
    \end{equation*}
    Since $|\s(X_{n})\s(q)|=\sqrt{\s(p)^2+1}$, we have
    \begin{equation*}
        \sqrt{\s(p)^2+1}<e^{-C_{s,\s}}+|\s(p)|
    \end{equation*}
    and
    \begin{equation}\label{p_bounds}
        |\s(p)|>\frac{e^{C_{s,\s}}-e^{-C_{s,\s}}}{2}.
    \end{equation}
    An easy calculation shows that
    \begin{equation}\label{03mainBound}
        \frac{2}{e^{C_{s,\s}}-e^{-C_{s,\s}}}=\left|\frac{s_{\s}\s(X_{n})-\s(a_s)}{\s(X_{n}+1)}\right|.
    \end{equation}
    In a similar manner to the proof of \cref{maintheorem2_explicit_12}, we obtain
    \begin{align*}
        &\left|\s \left(\frac{X_{n}^2q-1}{p}\right)-\sgn\left(\s\left(pq\right)\right)\s(X_{n})\right| \\
        &\leq \left|\s\left(\frac{X_{n}}{p}\right)\right|\left|\s(X_{n})\s(q)-\sgn\left(\s\left(pq\right)\right)\s(p)\right|+\left|\s\left(\frac{1}{p}\right)\right| \\
        &\leq\left|\s\left(\frac{1}{p}\right)\right|(\s(X_{n})+1).
    \end{align*}
    Using (\ref{p_bounds}) and (\ref{03mainBound}), we have
    \begin{equation*}
        \left|\s\left(\frac{X_{n}^2q-1}{p}\right)\right|<|s_{\s}\s(X_{n})-\s(a_s)|,
    \end{equation*}
    and the assertion follows.
\end{proof}

%% file: sec6_rel_to_pell_eq.tex
In \cite{Yoshizaki2023}, the second author pointed out that there is an interesting relationship between the continued fraction
\begin{equation}\label{12_type_yoshizaki}
  X_n=\left[1,\overline{\frac{1}{1+{X_{n-1}}},2}\right]
\end{equation}
and the ``generalized Pell equation'';
\begin{equation*}\label{genPell}
  x^2-X_n^2y^2=1.
\end{equation*}
We set $\text{Sol}_n^+=\{(x,y)\in \Z[X_{n-1}]^2 \mid x^2-X_n^2y^2=1\}$. 
Then the map $RE_n^+ \to \text{Sol}_n^+;\e \mapsto (p(\e), q(\e))$ is a bijection.
In the case of $n=1$, $\text{Sol}_1^+$ is the set of integer solutions of the Pell equation $x^2-2y^2=1$,
and it is well-known that $RE_1^+$ is generated by $-1$ and $3+2\sqrt{2}$.
Recall that the generator $3+2\sqrt{2}$ is obtained from the first convergent $3/2$ of $\sqrt{2} = [1, \overline{2}]$.
In a similar manner to the case $n=1$,
we expect that generators of $RE_n^+$ are also obtained from the $1$st convergent and its conjugates of (\ref{12_type_yoshizaki}).
In fact,
the second author showed that this expectation is correct under the Weber's conjecture (\cite[Conjecture 3.6 and Theorem 4.1]{Yoshizaki2023}).
Note that the Weber's conjecture states that the class number $h_n$ of $\B_n$ is $1$ for each $n>0$ and it is proved for $1\leq n \leq 6$ (cf. Miller~\cite[Section 2]{Miller2014}).

From \cref{maintheorem1_explicit}, 
we obtain a $(0,3)$-type $\mathbb{Z}[X_{n-1}]$-PCF expansion of $X_n$ for each $n\in\Z_{\geq 1}$.
If we relax the minimality of non-periodic part, 
we regard it as a $(1,3)$-type $\mathbb{Z}[X_{n-1}]$-PCF expansion of $X_n$
\begin{equation}\label{13pcfofXn}
  \left[\frac{(\eta_n-\eta_{n-1})X_n-1}{\eta_{n-1}},\overline{\eta_{n-1},\frac{(\eta_n-\eta_{n-1})/X_n-1}{\eta_{n-1}},\frac{(\eta_n-\eta_{n-1})X_n-1}{\eta_{n-1}}}\right].
\end{equation}
In this paper, 
we obtain a similar relationship between the continued fraction (\ref{13pcfofXn}) and the generalized Pell equations
\begin{equation*}\label{pmGenPell}
  x^2-X_n^2y^2=\pm 1.
\end{equation*}
We set $\text{Sol}_n=\{(x,y)\in \Z[X_{n-1}]^2 \mid x^2-X_n^2y^2 \in \{\pm1\}\}$.
Then the map $RE_n \to \text{Sol}_n;\e \mapsto (p(\e), q(\e))$ is a bijection.
Recall that all integer solutions of the Pell equations $x^2-X_1^2y^2 = \pm 1$ are generated by the fundamental unit $p+qX_1$, 
where $p/q = 1/1$ is the $\ell-1 = 0$th convergent of the regular continued fraction $X_1 = \sqrt{2} = [1, \overline{2}]$ of type $(N,\ell) = (1,1)$.
In a similar manner to the case $n=1$,
we expect that generators of $RE_n$ are also obtained from the $2$nd convergent of the continued fraction (\ref{13pcfofXn}).
More precisely, let $p_{n,2}/q_{n,2}$ be the $2$nd convergent of (\ref{13pcfofXn}) and set
\begin{equation}
  A_n=\langle-1, \sigma(p_{n,2}+X_nq_{n,2})\mid \sigma\in\Gal(\mathbb{B}_n/\Q)\rangle_{\Z} 
\end{equation}
for each $n>0$.
Here, $\langle S \rangle_{\Z}$ is the subgroup of $RE_n$ generated by a subset $S\subset RE_n$.
We obtain the following proposition.

\begin{proposition}\label{fundsolofXn}
  We obtain $RE_n=A_n$ for each $n>0$ under the Weber's conjecture.
\end{proposition}

\begin{proof}
We have
\[
\frac{p_{n,2}}{q_{n,2}}=\frac{\eta_{n-1}}{(\eta_n-\eta_{n-1})/X_n}    
\]
and we see that $p_{n,2}+X_{n}q_{n,2}=\eta_n$. 
Thus,
we obtain
\[
A_n=\langle-1, \sigma(\eta_n)\mid \sigma\in\Gal(\mathbb{B}_n/\Q)\rangle_{\Z}.
\]
In what follows,
we will show $(RE_n : A_n)=h_n/h_{n-1}$ for $n\in\Z_{\geq 1}$, where $(RE_n : A_n)$ is the index.
Set $A_n^+=A_n\cap RE_n^+$.
By Morisawa and Okazaki~\cite[Lemma 3.2, (2)]{MorisawaOkazaki2020}, 
we see that $(RE_n : A_n)=(RE_n^+ : A_n^+)$.
Hence it is sufficient to show $(RE_n^{+} : A_n^{+})=h_n/h_{n-1}$,
which is already shown in Yoshizaki~\cite[Section 4]{Yoshizaki2023}.
For convenience,
we recall the outline of the proof of $(RE_n^{+} : A_n^{+})=h_n/h_{n-1}$.

Let $C_n$ be the group of cyclotomic units in $\mathbb{B}_n$ (cf.\ Washington~\cite[Chapter 8]{Washington}).
Then we have $A_n^+=RE_n^+\cap C_n$ and the relative norm induces the following exact sequence:
\[
1\rightarrow RE_n^+/A_n^+ \rightarrow \mathbb{Z}[X_n]^{\times}/C_n \rightarrow \mathbb{Z}[X_{n-1}]^{\times}/C_{n-1} \rightarrow 1.
\]
Since $(\Z[X_n]^{\times} : C_n)=h_n$ (cf.\ Washington~\cite[Theorem 8.2]{Washington}), 
we obtain $(RE_n^+ : A_n^+)=h_n/h_{n-1}$.
\end{proof}

%% file: main.bbl
\begin{thebibliography}{Was97}

\bibitem[BEJ21]{BrockElkiesJordan2021}
Bradley~W. Brock, Noam~D. Elkies, and Bruce~W. Jordan.
\newblock Periodic continued fractions over $s$-integers in number fields and
  skolem's $p$-adic method.
\newblock {\em Acta Arithmetica}, 197(4):379--420, 2021.

\bibitem[Mil14]{Miller2014}
John~C. Miller.
\newblock Class numbers of totally real fields and applications to the weber
  class number problem.
\newblock {\em Acta Arithmetica}, 164:381--397, 2014.

\bibitem[MO20]{MorisawaOkazaki2020}
T.~Morisawa and R.~Okazaki.
\newblock Filtrations of units of {V}i\`{e}te field.
\newblock {\em International Journal of Number Theory}, 16(05):1067--1079,
  2020.

\bibitem[Was97]{Washington}
Lawrence~C. Washington.
\newblock {\em Introduction to cyclotomic fields}, volume~83 of {\em Graduate
  Texts in Mathematics}.
\newblock Springer-Verlag, New York, second edition, 1997.

\bibitem[Yos23]{Yoshizaki2023}
H.~Yoshizaki.
\newblock Generalized {P}ell's equations and {W}eber's class number problem.
\newblock {\em Journal de Th\'{e}orie des Nombres de Bordeaux}, 35(2):pp.
  373--391, 2023.

\end{thebibliography}
